\documentclass[12pt]{amsart}
\usepackage{amssymb}
\usepackage{mathtools}
\usepackage[english]{babel}
\usepackage{epsfig}
\usepackage{mathptmx}
\usepackage{amsmath,amsfonts,amssymb}
\usepackage{mathtools}
\usepackage{mathrsfs}
\usepackage[usenames]{color}
\usepackage[all]{xy}
\usepackage{graphicx}
\usepackage{latexsym}
\usepackage{verbatim}

\numberwithin{equation}{section}

\newcommand{\R}{\mathbb R}

\newcommand{\C}{\mathbb C}

\newcommand{\N}{\mathbb N}

\def\1#1{\overline{#1}}
\def\2#1{\widetilde{#1}}
\def\3#1{\widehat{#1}}
\def\4#1{\mathbb{#1}}
\def\5#1{\frak{#1}}
\def\6#1{{\mathcal{#1}}}

\newcommand{\mcite}[1]{\csname b@#1\endcsname}

\theoremstyle{theorem}

\newtheorem{theorem}{Theorem}[section]
\newtheorem{lemma}[theorem]{Lemma}
\newtheorem{proposition}[theorem]{Proposition}

\theoremstyle{definition}
\newtheorem{definition}[theorem]{Definition}

\theoremstyle{remark}

\numberwithin{equation}{section}

\begin{document}

\title{Frames of quasi-geodesics, visibility and   geodesic loops}
\author[Filippo Bracci]{Filippo Bracci}

\address{F. Bracci: Dipartimento Di Matematica\\
Universit\`{a} di Roma \textquotedblleft Tor Vergata\textquotedblright\ \\
Via Della Ricerca Scientifica 1, 00133 \\
Roma, Italy}
\email{fbracci@mat.uniroma2.it}

\thanks{Partially supported by PRIN 2017 Real and Complex Manifolds: Topology, Geometry and holomorphic
dynamics, Ref: 2017JZ2SW5, by GNSAGA of INdAM and by the MUR Excellence Department Project MatMod@TOV CUP:E83C23000330006 awarded
to the Department of Mathematics, University of Rome Tor Vergata}

\date{\today}

\begin{abstract}
In this paper we give a characterization in terms of ``quasi-geodesics frames' of visibility and existence of geodesic loops for bounded domains in $\C^d$ which are Kobayashi complete hyperbolic and Gromov hyperbolic. 
\end{abstract}

\dedicatory{Dedicated to Prof. Simeon Reich, on the occasion of his 75th birthday}

\maketitle
\tableofcontents
\section{Introduction}
The concept of visibility in several complex variables for bounded and unbounded domains in $\C^d$ has been recently introduced and turned out to be a key notion to study continuous extension of biholomorphisms, estimates for the Kobayashi distance and iteration theory (we refer the reader to, {\sl e.g.}, \cite{BNT, BG0, BGNT, BM, BZ, BZ2, CMS, Ma, NOT, NO}). 

In this note we consider $\Omega\subset \C^d$ a bounded domain (with no assumption on the regularity of $\partial\Omega$) and we assume that its Kobayashi distance $K_\Omega$ is complete (see, {\sl e.g.}, \cite{Kob} for definition and properties of the Kobayashi distance). 

The domain $\Omega$ is {\sl visible} if, roughly speaking, the geodesics which converge to different points in  the boundary bend inside the domain. 

Due to the completeness of $K_\Omega$ and by Arzel\`a-Ascoli's theorem, a non-compactly divergent sequence of geodesics rays or lines admits a subsequence converging {\sl uniformly on compacta} to a geodesics ray or line.

One of the basic observations of the present paper is that, if $\Omega$ is visibile then the convergence is actually {\sl uniform} (not just on compacta)---see Lemma~\ref{Lem:uniform-conv-visibile}. 

Moreover, visibile domains have a natural family of geodesics which exhibit certain peculiarities (they form a kind of ``frame''). Motivated by this, we give the following definition:

\begin{definition}\label{Def:frame}
Let $\Omega\subset\subset \C^d$ be a domain such that $(\Omega, K_\Omega)$ is complete hyperbolic and let $A\geq 1$, $B\geq 0$.  A family $\mathcal F$  is a {\sl frame of $(A,B)$-quasi-geodesics} if
\begin{enumerate}
\item for every $\gamma\in\mathcal F$ there exists $R_\gamma\in (0,+\infty]$ such that, if $R_\gamma<+\infty$, $\gamma:[0, R_\gamma]\to \Omega$  is a $(A,B)$-quasi-geodesic segment, while, if $R_\gamma=+\infty$, $\gamma:[0, +\infty)\to \Omega$  is a $(A,B)$-quasi-geodesic ray,
\item if $\gamma\in\mathcal F$ is a quasi-geodesic ray ({\sl i.e.}, $R_\gamma=+\infty$), then there exists $p_\gamma\in\partial \Omega$ such that $\lim_{t\to+\infty}\gamma(t)=p_\gamma$,
\item there exists a compact subset $K\subset\Omega$ such that $\gamma(0)\in K$ for every $\gamma\in \mathcal F$,
\item for every sequence $\{\gamma_k\}\subset \mathcal F$  there exist a subsequence $\{\gamma_{k_m}\}$ and  $\gamma\in\mathcal F$ such that $\{\gamma_{k_m}\}$ converges uniformly  to $\gamma$,
\item there exist $\epsilon>0$ and $\delta>0$ such that for every $z\in \Omega$ with $\hbox{dist}(z,\partial\Omega)<\epsilon$ there exists  $\gamma\in\mathcal F$ such that $K_\Omega(z, \gamma([0,+\infty])<\delta$.
\end{enumerate}
\end{definition}
It turns out that visibile domains have a frame of geodesics (see Proposition~\ref{Prop:vis-frame}). One of the main result of the paper is to prove the converse for Gromov hyperbolic domains:

\begin{theorem}\label{thm:char-vis-frame}
Let $\Omega\subset\subset \C^d$ be a domain such that $(\Omega, K_\Omega)$ is complete hyperbolic and Gromov hyperbolic. Assume that $\partial \Omega$ does not contain nontrivial analytic discs. Then $\Omega$ is visible if and only if there exists a frame of $(A,B)$-quasi-geodesics for some $A\geq 1, B\geq 0$.
\end{theorem}

The hypothesis on the non-existence of nontrivial analytic discs on the boundary is technical. It is presently not known whether a visible complete hyperbolic and Gromov hyperbolic domain might have nontrivial analytic discs on the boundary (in \cite[Theorem~1.2]{BZ2} it is shown that for a visibile complete hyperbolic and Gromov hyperbolic domain with no geodesic loops--see below for the definition---and with $C^0$-boundary there can not be nontrivial analytic discs on the boundary).

The domain $\Omega$ is a {\sl Gromov model domain} provided the identity map extends as a homeomorphism from the Gromov compactification of $\Omega$ to the Euclidean closure $\overline{\Omega}$ (see \cite{BG0}). 

Due to the equivalence between Gromov's topology and Carath\'eodory's prime end topology (see, {\sl e.g.}, \cite[Ch. 4]{BCD}) in dimension one, Gromov model domains play essentially the r\^ole of simply connected Jordan domains, while, visible domains play the r\^ole of simply connected domains with locally connected boundary  with respect to the extension to the boundary of biholomorphisms (see \cite{BG0} for a careful explanation of this fact). 

The following are examples of Gromov model domains: $C^2$-bounded strongly pseudoconvex domains \cite{BB}, Gromov hyperbolic (bounded or unbounded) convex domain \cite{BG1, BG2, BGZ}, smooth bounded D'Angelo finite type convex domains \cite{Z1, Z2}, smooth bounded D'Angelo finite type pseudoconvex domains in $\C^2$ \cite{Fia}, bounded Gromov hyperbolic Lipschitz $\C$-convex domains \cite{BGNT}.

It turns out (see \cite{BG0, BGNT}) that a bounded Kobayashi complete hyperbolic and Gromov hyperbolic domain $\Omega$ is a Gromov model if and only if $\Omega$ is visible and has no {\sl geodesic loops}.  Where, a geodesic loop for $\overline{\Omega}$ is a geodesic line $\gamma:(-\infty,+\infty)\to \Omega$ such that the cluster set of $\gamma$ at $+\infty$ coincides with the cluster set of $\gamma$ at $-\infty$.

In this direction, we first show (see Proposition~\ref{Prop:visibi-no-cont}) that if $\Omega$ is visible and there exist $p\in\partial\Omega$ and $U$ an open neighborhood of $p$ such that $U\cap \Omega$ has at least two connected components having $p$ on their boundary then there exists a geodesic loop for $\overline{\Omega}$.

The second main result of the paper is a characterization of (non)existence of geodesic loops in terms of frames of quasi-geodesics. 

\begin{definition}
Let $\Omega\subset\subset \C^d$ be a domain such that $(\Omega, K_\Omega)$ is complete hyperbolic and let $A\geq 1$, $B\geq 0$.  A frame of $(A,B)$-quasi-geodesics $\mathcal F$ is  {\sl  looping} if  there exist two  quasi-geodesic rays $\gamma, \eta\in\mathcal F$ such that $\lim_{t\to+\infty}\gamma(t)=\lim_{t\to+\infty}\eta(t)$ and which stay at infinite distance each other. We say that the quasi-geodesic frame $\mathcal F$ is {\sl non-looping} otherwise. 
\end{definition}

Then we prove:

\begin{theorem}\label{Thm:char-geo-loop}
Let $\Omega\subset\subset \C^d$ be a domain such that $(\Omega, K_\Omega)$ is complete hyperbolic and Gromov hyperbolic. Assume $\partial \Omega$ does not contain nontrivial analytic discs and that $\Omega$ is visible. Then $\Omega$ has no geodesic loops if and only if it has a non-looping frame of $(A,B)$-quasi-geodesics for some $A\geq 1$ and $B\geq 0$. 
\end{theorem}

As a consequence of Theorem~\ref{thm:char-vis-frame}, Theorem~\ref{Thm:char-geo-loop} and the previous discussion we have:

\begin{theorem}\label{frame-Gromovo}
Let $\Omega\subset\subset \C^d$ be a domain such that $(\Omega, K_\Omega)$ is complete hyperbolic and Gromov hyperbolic. Assume $\partial \Omega$ does not contain nontrivial analytic discs. Then $\Omega$ is a Gromov model domain if and only if it has a non-looping frame of $(A,B)$-quasi-geodesics for some $A\geq 1$ and $B\geq 0$.
\end{theorem}

The definition of frames of (non-looping) quasi-geodesics  might seem at a first sight rather technical and useless. However, unravelling the proofs in \cite{BG1, BG2, BGZ, BGZ1, Z1, Z2, Z3, Fia} where several kind of (Gromov hyperbolic) bounded domains are proved to be  Gromov models, one sees that actually the main work was exactly to construct non-looping quasi-geodesic frames  in the sense of our definition. As another example, we sketch an argument for constructing a frame of non-looping quasi geodesics in strongly pseudoconvex domains, taking for granted that a $C^2$ bounded strongly pseudoconvex domain $\Omega$ is Kobayashi hyperbolic and Gromov hyperbolic. Indeed, one can construct easily a frame of non-looping  quasi-geodesics as follows. If $p\in\partial\Omega$ and $U_p$ is an open neighborhood of $p$ such that $\Omega\cap U_p$ is biholomorphic to a strongly convex domain $V_p$, then one can consider the family of all real segments in $V_p$ steaming from a fixed point in $V_p$. Arguing as in \cite{BG1, BG2} one can see that this  family is a frame of non-looping  quasi-geodesics in $V_p$. Hence, its preimage $\mathcal F_p$ is  a frame of non-looping  quasi-geodesics  in $\Omega\cap U_p$. Taking $\mathcal F$ to be the union of all $\mathcal F_p$ and using the localization of the Kobayashi distance at the boundary, one can show that $\mathcal F$ is a frame of non-looping quasi-geodesics for $\Omega$ and hence, by Theorem~\ref{frame-Gromovo}, $\Omega$ is a Gromov model domain (a result well known, as remarked above, after \cite{BB}). 

The paper is organized as follows. In Section~\ref{S2} we recall some definitions and state some preliminary results we need in the proofs. In Section~\ref{S3} we prove Theorem~\ref{thm:char-vis-frame} and in Section~\ref{S4} we prove Theorem~\ref{Thm:char-geo-loop}.

\section{Notations and preliminaries}\label{S2}

In this section we let $\Omega \subset \mathbb C^d$ be a bounded domain.  We denote by $k_{\Omega}$ the infinitesimal Kobayashi pseudometric of $\Omega$ and by $K_{\Omega}$ the Kobayashi distance of $\Omega$. We refer the reader to \cite{Kob} for definitions and properties.

If $\Omega \subset \subset \mathbb C^d$ is complete hyperbolic, that is, $(\Omega,K_{\Omega})$ is a complete metric space, it follows from the Hopf-Rinow theorem that $(\Omega,K_{\Omega})$ is geodesic and thus, every couple of points in $\Omega$ can be joined by a geodesic ({\sl i.e.} length minimizing curve) for $K_\Omega$. If $p, q\in \Omega$, we denote by $[p,q]_\Omega$ any geodesic joining $p$ and $q$.

A geodesic $\gamma:[a,b]\to \Omega$, $-\infty<a<b<+\infty$ is called a {\sl geodesic segment}. A geodesic $\gamma:[a,+\infty)\to \Omega$, $a\in \R$, is a {\sl geodesic ray} and a geodesic $\gamma:(-\infty,+\infty)\to \Omega$ is a {\sl geodesic line}.

A geodesic triangle $T$ is the union of 3 geodesic segments (called {\sl sides}) $T=[x,y]_\Omega \cup [y,z]_\Omega \cup [z,x]_\Omega$ joining 3 points $x,\ y, z \in X$.

The complete  metric space $(\Omega,K_\Omega)$ is {\sl Gromov hyperbolic} if there exists $\delta > 0$ (the {\sl Gromov constant} of $\Omega$)
such that  every geodesic triangle $T$
is $\delta$-thin, that is,
every point on a side of $T$ has distance from the union of the other two sides less than or equal to $\delta$ (see, {\sl e.g.}, \cite{BH, Cor-Del-Pap} for details and further properties).

For an absolutely continuous curve $\gamma:[a,b] \rightarrow \Omega$, we denote by $l_{k_\Omega}(\gamma;[s,t])$  the length of the curve $\gamma$ on $[s,t]$, $a\leq s<t\leq b$, that is
\[
l_{k_\Omega}(\gamma;[s,t]):=\int_s^t k_\Omega(\gamma(\tau); \gamma'(\tau))d\tau.
\]
Let $A > 1$ and $B > 0$. An absolutely continuous curve $\gamma:[a,b] \rightarrow \Omega$ is  a {\sl $(A,B)$-quasi-geodesic} if for every $a \leq s < t \leq b$, we have :
$$
 l_{k_\Omega}(\gamma;[s,t]) \leq A K_\Omega(\gamma(s),\gamma(t)) + B.
$$
An $(A,B)$-quasi-geodesic ray (or line) is an absolutely continuous curve whose restriction to any compact interval in the domain of definition is an $(A,B)$-quasi-geodesic.
A {\sl quasi-geodesic} is just any $(A,B)$-quasi-geodesic segment/ray/line for some $A\geq 1$ and $B\geq 0$.

One of the main feature of Gromov hyperbolic spaces is the so-called Geodesic Stability Theorem, which says that every $(A,B)$-quasi-geodesic is shadowed by a geodesic at a distance which depends only on $A,B$ and the Gromov constant of the space. In this paper we do not need it directly, but we will use instead a straightforward consequence  for ``quasi-geodesic rectangles'' (see, {\sl e.g.}, \cite[Observation 4.4]{Z2}):

\begin{lemma}\label{Lem:thick-rect}
Let $\Omega\subset\subset\C^d$ be a Kobayashi complete hyperbolic domain such that $(\Omega, K_\Omega)$ is Gromov hyperbolic. Let $a, b, c, d\in \Omega$. Let $A\geq 1$ and $B\geq 0$. If $\Gamma_1$ is a $(A,B)$-quasi-geodesic joining $a$ with $b$, $\Gamma_2$ is a $(A,B)$-quasi-geodesic joining $b$ with $c$, $\Gamma_3$ is a $(A,B)$-quasi-geodesic joining $c$ with $d$, $\Gamma_4$ is a $(A,B)$-quasi-geodesic joining $a$ with $d$ (that is, $\Gamma_1\cup \Gamma_2\cup \Gamma_3\cup \Gamma_4$ is a {\sl $(A,B)$-quasi-geodesic rectangle} with sides $\Gamma_1, \Gamma_2, \Gamma_3$ and $\Gamma_4$) then there exists $N>0$ (which depends only on $A, B$ and the Gromov constant of $\Omega$) such that every point of one side is contained in the $N$-tubular neighborhood (with respect to $K_\Omega$) of the union of the other three.
\end{lemma}

We now turn to the precise definition of visibility.

Let $p, q\in \partial \Omega$,  $p\neq q$. 
We say that the couple $(p,q)$ satisfies the {\sl visibility condition with respect to $K_\Omega$} 
if  there exist a neighborhood 
$V_p$ of $p$ and a neighborhood $V_q$ of $q$ and a compact subset $K$ of $\Omega$ such that $V_p \cap V_q = \emptyset$ and $[x,y]_{\Omega} \cap K \neq \emptyset$ for every $x \in V_p\cap \Omega$, $y \in V_q\cap \Omega$.

We say that $\Omega$ is {\sl visible}  if  every couple of points $p,q \in \partial \Omega$, $p \neq q$, 
satisfies the visibility condition with respect to $K_\Omega$.

We say that a geodesic line $\gamma:(-\infty,+\infty)\to \Omega$ is a {\sl geodesic loop in $\overline{D}$} if $\gamma$ has the same cluster set $\Gamma$ in $\overline{\Omega}$ at $+\infty$ and $-\infty$. In such a case we say that $\Gamma$ is the {\sl vertex} of the geodesic loop~$\gamma$.

By \cite[Lemma 3.1]{BNT}, we have:

\begin{lemma}\label{lem:vis-land}
Suppose $\Omega\subset\C^d$ is a bounded complete hyperbolic domain. If $\Omega$ is visible then every geodesic ray lands, {\sl i.e.}, if $\gamma:[0,+\infty)\to\Omega$ is a geodesic then there exists $p\in\partial\Omega$ such that $\lim_{t\to+\infty}\gamma(t)=p$. In particular, the vertex of every geodesic loop in $\overline{\Omega}$ is a point in $\partial\Omega$.
\end{lemma}

In the sequel, we will also need the following lemma (see \cite[Lemma A.2]{BG0})

\begin{lemma}[D'Addezio]\label{Daddezio}
Let $D\subset \C^d$ be a complete hyperbolic bounded domain. Assume that
$\partial D$ does not contain non-trivial analytic discs. If $\{z_n\}, \{w_n\} \subset D$ are two sequences such
that $\lim_{n\to+\infty} z_n = p, \lim_{n\to+\infty} w_n = q$ with $p,q\in\partial D$ and
\[
\sup_n K_D(z_n,w_n)<+\infty,
\]
then $p = q$.
\end{lemma}

\section{Visible domains and frame of quasi-geodesics}\label{S3}

As a matter of notation,  if $\{\gamma_k\}$ is a sequence of curves in $\C^d$ such that for every $k$ the curve $\gamma_k$ is defined on some interval $[0,R_k]$, with $R_k\in (0,+\infty]$,  we say that $\{\gamma_k\}$ converges {\sl uniformly} to a curve $\gamma:[0, a]\to \Omega$ provided that $a=\lim_{k\to+\infty}R_k$ and
for every $\theta>0$ there exists $k_0\in\N$ such that for every $k\geq k_0$ and for every $t\in [0,R_k]$ it holds
\[
|\gamma_k(t)-\gamma(t)|\leq \theta.
\]

\begin{lemma}\label{Lem:uniform-conv-visibile}
Let $\Omega\subset\subset \C^d$ be a domain such that $(\Omega, K_\Omega)$ is complete hyperbolic. Suppose $\Omega$ is visible. If $\{\gamma_k\}$ is a sequence of geodesic (segment or rays) in $\Omega$ which converges uniformly on compacta to a geodesic ray $\gamma$ then $\{\gamma_k\}$ converges uniformly to $\gamma$.
\end{lemma}
\begin{proof}
By hypothesis, for every $k$, either  $\gamma_k$ is a geodesic segment defined on $[0,R_k]$ for some $R_k>0$ or $\gamma_k$ is geodesic ray defined on $[0,+\infty)$. 

In case $\gamma_k$ is a geodesic ray, since $\Omega$ is visible, by Lemma~\ref{lem:vis-land}  there exists $\gamma_k(+\infty):=\lim_{t\to+\infty}\gamma_k$, hence, in this case, in order to unify notations, we can consider $\gamma_k$ to be defined on $[0,+\infty]$ and set $R_k=+\infty$. 

Now, assume by contradiction there exists $\theta>0$ and a sequence $\{t_k\}$ such that $t_k\in [0,R_k]$ and
\[
|\gamma_k(t_k)-\gamma(t_k)|\geq \theta,
\]
where, as before, if $t_k=+\infty$ we set $\gamma(+\infty):=\lim_{t\to+\infty}\gamma(t)$.

Since $\{\gamma_k\}$ converges uniformly on compacta to $\gamma$, it follows that either $t_k=+\infty$ or $\{t_k\}$ converges to $+\infty$. Hence, up to subsequences, we can assume that $\{\gamma_k(t_k)\}$ converges to a point $q\in\partial \Omega$ such that
\[
|q-\gamma(+\infty)|\geq \theta.
\]
Fix an open neighborhood $U$ of $\gamma(+\infty)$ and an open neighborhood $V$ of $q$ such that $\overline{U}\cap\overline{V}=\emptyset$. By visibility hypothesis, there exists a compact subset $K\subset\Omega$ such that every geodesic in $\Omega$ joining a point of $U\cap \Omega$ to a point in $V\cap \Omega$ has to intersect $K$.

Since $\gamma(t)$ tends to $\gamma(+\infty)$ as $t\to+\infty$, there exists $R>0$ such that $\gamma(r)\in U\cap\Omega$ for all $r\geq R$. Hence, for every $r\geq R$, there exists $k_r\in\N$ such that $\gamma_k(r)\in U\cap \Omega$ for $k\geq k_r$. Up to take $k_r$ larger, we can also assume that $\gamma_k(t_k)\in V\cap\overline{\Omega}$ for $k\geq k_r$.

Thus,  $\gamma_k([r,t_k))$ is a geodesic from $U\cap \Omega$ to $V\cap \Omega$ for $k\geq k_r$. Therefore, for every $k\geq k_r$ there exists $s_k\in (r,t_k)$ such that $\gamma_k(s_k)\in K$. 

But,
\[
+\infty>\max_{\xi\in K}K_\Omega(z_0, \xi))\geq K_\Omega(z_0, \gamma_k(s_k))=K_\Omega(\gamma_k(0), \gamma_k(s_k))=s_k>r,
\]
which gives a contradiction for $r$ large enough. 
\end{proof}

\begin{proposition}\label{Prop:vis-frame}
Let $\Omega\subset\subset \C^d$ be a domain such that $(\Omega, K_\Omega)$ is complete hyperbolic. If $\Omega$ is visible then there exists a frame of geodesics.
\end{proposition}
\begin{proof}
Fix $z_0\in\Omega$ and let $\mathcal F$ be the set of all geodesic steaming from $z_0$, {\sl i.e.}, $\gamma\in\mathcal F$ if either there exists $R_\gamma\in (0,+\infty)$ such that $\gamma:[0,R_\gamma]\to \Omega$ is a geodesic segment or $\gamma:[0,+\infty)\to \Omega$ is a geodesic ray, and (in both cases) $\gamma(0)=z_0$. By definition, $\mathcal F$ satisfies condition (1) and (3) of Definition~\ref{Def:frame}. Also, since the distance $K_\Omega$ is complete, by Hopf-Rinow's theorem, for every $w\in \Omega$ there exists a geodesic segment $\gamma$ joining $z_0$ and $w$, hence, (5) is trivially satisfied for any $\epsilon>0$ and $\delta>0$.

Moreover, since $\Omega$ is visible, it follows from Lemma~\ref{lem:vis-land} that $\mathcal F$ enjoys also (2).

Now, we show that $\mathcal F$ satisfies (4).  Assume $\{\gamma_k\}$ is a sequence of geodesic steaming from $z_0$. 

Note that for every $k$, either  $\gamma_k$ is a geodesic segment defined on $[0,R_k]$ for some $R_k>0$ or $\gamma_k$ is geodesic ray defined on $[0,+\infty)$. However, as remarked above, in this case there exists $\gamma_k(+\infty):=\lim_{t\to+\infty}\gamma_k$, hence, we can consider $\gamma_k$ to be defined on $[0,+\infty]$ and set $R_k=+\infty$. 

By Arzel\`a-Ascoli's theorem, it follows that, up to subsequences, $\{\gamma_k\}$ converges {\sl uniformly on compacta} to some $\gamma\in\mathcal F$. By Lemma~\ref{Lem:uniform-conv-visibile} the convergence is actually uniform, and we are done. \end{proof}

Now we are in a good shape to prove Theorem~\ref{thm:char-vis-frame}:

\begin{proof}[Proof of Theorem~\ref{thm:char-vis-frame}]

One direction follows from Proposition~\ref{Prop:vis-frame}.

Conversely, Assume that $\Omega$ has a family $\mathcal F$ of $(A,B)$-quasi-geodesics. Hence, there exists a compact subset $K\subset \Omega$ such that every quasi-geodesic  in $\mathcal F$ steams from $K$. Assume $\{z^j_k\}\subset \Omega$ be a sequence converging to $p_j\in\partial \Omega$, $j=1,2$ with $p_1\neq p_2$. 

We argue by contradiction and we assume that $\{[z^1_k,z^2_k]_\Omega\}$ eventually escapes every compact subset of $\Omega$. By hypothesis, for $k$ large enough, we can find $\gamma_k^1, \gamma_k^2\in\mathcal F$  and $w_k^j\in\gamma_k^j$ such that  $K_\Omega(z_k^j, w_k^j)< \delta$, $j=1,2$ (here, we a slight abuse of notation, we denote by $\gamma_k^j$ also the image of the quasi-geodesic $\gamma_k^j$). 

By Lemma~\ref{Daddezio}, since $\partial \Omega$ does not contain nontrivial analytic discs, it follows that $\{w_k^j\}$ converges to $p_j$, $j=1,2$ for $k\to+\infty$. 

By condition (4) of Definition~\ref{Def:frame}, we can assume, up to subsequences, that $\{\gamma_k^j\}$ converges uniformly to some $\gamma^j\in\mathcal F$, $j=1,2$. Hence, 
\begin{equation}\label{Eq:geo-fram-g}
\lim_{t\to+\infty}\gamma^j(t)=p_j, \quad j=1,2.
\end{equation}

We claim now that $\{[w^1_k,w^2_k]_\Omega\}$ eventually escapes every compact subset of $\Omega$, if so does $\{[z^1_k,z^2_k]_\Omega\}$. Indeed, assume by contradiction that there exists a compact set $Q\subset\subset \Omega$ and $c_k\in [w^1_k,w^2_k]_\Omega$ such that $c_k\in Q$ for every $k$. Since 
\[
[w^1_k,w^2_k]_\Omega\cup [z^1_k,z^2_k]_\Omega\cup [z^1_k,w^1_k]_\Omega\cup [z^2_k,w^2_k]_\Omega
\]
is a geodesic rectangle, by Lemma~\ref{Lem:thick-rect}, it follows that there exist $N>0$ and  
\[
t_k\in [z^1_k,z^2_k]_\Omega\cup [z^1_k,w^1_k]_\Omega\cup [z^2_k,w^2_k]_\Omega
\]
such that $K_\Omega(t_k, c_k)\leq N$. Since for every $s\in [z^j_k,w^j_k]_\Omega$ we have that $K_\Omega(s, z^j_k)\leq K_\Omega(w^j_k, z^j_k)< \delta$, $j=1,2$, it follows from the completeness of $K_\Omega$ that $t_k\in  [z^1_k,z^2_k]_\Omega$. But then $t_k$ belongs to a $N$-tubular neighborhood of $Q$, which is as well a relatively compact subset of $\Omega$, and we get a contradiction.

Therefore, $\{[w^1_k,w^2_k]_\Omega\}$ eventually escapes every compact subset of $\Omega$.

Let $a^j_k>0$ be such that $\gamma^j_k(a_k^j)=w_k^j$, $j=0,1$ and consider now the quasi-geodesic rectangle
\[
[w^1_k,w^2_k]_\Omega \cup \gamma_k^1([0,a_k^1])\cup \gamma_k^2([0,a_k^1])\cup [\gamma_k^1(0), \gamma_k^2(0)]_\Omega.
\]
Note that $\{[\gamma_k^1(0), \gamma_k^2(0)]_\Omega\}$ is relatively compact in $\Omega$. Hence, by Lemma~\ref{Lem:thick-rect}, there exists $N>0$ such that for every point $s_k\in [w^1_k,w^2_k]_\Omega$ there exists $\zeta(s_k)\in \gamma_k^1([0,a_k^1])\cup \gamma_k^2([0,a_k^1])$ such that 
\begin{equation}\label{Eq:close-geo-goF}
K_\Omega(s_k, \zeta(s_k))\leq N.
\end{equation}
Since $\{s_k\}$ is compactly divergent, it follows by the completeness of $K_\Omega$ that $\{\zeta(s_k)\}$ is compactly divergent as well. By \eqref{Eq:geo-fram-g}, the cluster set of $\{\zeta(s_k)\}$ is contained in $\{p_1\}\cup\{p_2\}$ and, by \eqref{Eq:close-geo-goF} and Lemma~\ref{Daddezio}, so is the cluster set of $\{s_k\}$. But this is a contradiction because the cluster set of $[w^1_k,w^2_k]_\Omega$ has to be connected and contains both $p_1$ and $p_2$. Hence, $\Omega$ is visible.
\end{proof}

\section{Geodesic loops in visible domains}\label{S4}

\begin{definition}
Let $\Omega\subset\subset \C^d$ be a domain such that $(\Omega, K_\Omega)$ is complete hyperbolic. We say that two quasi-geodesic rays $\gamma, \eta$ {\sl stay at infinite distance each other} provided for every $C>0$ there exists $t_C>0$ such that
\[
\min_{t\geq t_C} K_\Omega(\gamma(t), \eta([0,+\infty))>C, \quad \min_{t\geq t_C} K_\Omega(\eta(t), \gamma([0,+\infty))>C.
\]
\end{definition}

\begin{lemma}\label{Lem:infinite-dist-ray}
Let $\Omega\subset\subset \C^d$ be a domain such that $(\Omega, K_\Omega)$ is complete hyperbolic. Assume $\gamma:\R\to\Omega$ is a geodesic loop for $\overline{D}$. Let $\gamma^+:[0,+\infty)\to\Omega$ be the geodesic ray defined by $\gamma^+(t):=\gamma(t)$ and let $\gamma^-:[0,+\infty)\to\Omega$ be the geodesic ray defined by $\gamma^-(t):=\gamma^-(t)$. Then $\gamma^+$ and $\gamma^-$ stay at infinite distance each other.
\end{lemma}
\begin{proof}
Assuming by contradiction that $\gamma^+, \gamma^-$ do not stay at infinite distance each other, we can find $C>0$ such that for every $t\in [0,+\infty)$ there exists $s_t\in [0,+\infty)$ such that 
\[
K_\Omega(\gamma^+(t), \gamma^-(s_t))\leq C.
\]
But, since $\gamma$ is a geodesic,
\[
t+s_t=K_\Omega(\gamma(t),\gamma(-s_t))=K_\Omega(\gamma^+(t), \gamma^-(s_t))\leq C,
\]
and, for $t\to +\infty$ we have a contradiction.
\end{proof}

\begin{proposition}\label{Prop:visibi-no-cont}
Let $\Omega\subset\subset \C^d$ be a domain such that $(\Omega, K_\Omega)$ is complete hyperbolic. Assume that $\Omega$ is visible. If there exists $p\in\partial \Omega$ and an open neighborhood $U$ of $p$ such that $U\cap \Omega$ has at least two connected components such that $p$ belongs to their closure, then there exists a geodesic loop for $\overline{\Omega}$ with vertex  $p$.
\end{proposition}
\begin{proof}
Let $p, U$ as in the statement and let $V_1, V_2$ be two connected components of $\Omega\cap U$ such that $p\in \overline{V_j}$, $j=1,2$. For $j=1,2$, let $\mathcal V_j\subset\subset V_j$ be an open set such that $p\in \overline{\mathcal{V}_j}$ and $p\not\in \partial\mathcal V_j$. Moreover, for $j=1,2$, let $\{z^j_n\}\subset \mathcal V_j$ be a sequence converging to $p$. For every $n$,  let $\gamma_n: [a_n, b_n]:\to \Omega$ be a geodesic such that $\gamma_n(a_n)=z^1_n$ and $\gamma_n(b_n)=z^2_n$, $a_n<b_n$. 

We claim that there exists a compact subset $K\subset \Omega$ such that $\gamma_n([a_n, b_n])\cap K\neq\emptyset$ for all $n$. Indeed, if this were not the case, we can assume that $\{\gamma_n([a_n, b_n])\}$ escapes every compact subset of $\Omega$ for $n$ large. Let $r_n\in (a_n, b_n)$ be such that $\gamma_n(r_n)\in \partial \mathcal V_2$, for all $t$. Then, up to subsequences, $\{\gamma_n(r_n)\}$ converges, for $n\to\infty$, to a point $q\in\partial\Omega\setminus\{p\}$. Since $\Omega$ is visible, it follows that there exists a compact subset $K'\subset\Omega$ such that $\gamma_n([r_n, b_n])\cap K'\neq\emptyset$, reaching a contradiction. 

Up to an affine change of parameterization, we can thus assume that $\{\gamma_n(0)\}$ is relatively compact in $\Omega$ and $a_n<0<b_n$ for all $n$. 

We claim that $\{\gamma_n\}$ (under the assumption that $\{\gamma_n(0)\}$ is relatively compact in $\Omega$) converges, up to subsequences, to a geodesic loop in $\overline{\Omega}$ with vertex  $p$. 

Indeed, since $(\Omega, K_\Omega)$ is complete and $\gamma_n(0)\in K$,
\[
\lim_{n\to\infty} |a_n|=\lim_{n\to\infty} K_\Omega(\gamma_n(0),z_n^1)=+\infty,
\]
and thus $a_n\to-\infty$ for $n\to \infty$. Similarly, $b_n\to +\infty$ for $n\to\infty$. Thus, for every $S,T\in \R$ there exists $n_0$ such that  $S,T\in(a_n, b_n)$ for $n>n_0$. Moreover, setting  for $R>0$, 
\[
d_K(R):=\max\{K_\Omega(z,w): K_\Omega(z,K)\leq R, K_\Omega(w,K)\leq R\}<+\infty,
\]
we have, for all  $n>n_0$,
\[
K_\Omega(\gamma_n(S), \gamma_n(T))= K_\Omega(\gamma_n(0),\gamma_n(S))+K_\Omega(\gamma_n(0),\gamma_n(T))\leq d_K(|S|)+d_K(|T|).
\]
Hence $\{\gamma_n\}$ are locally equibounded and locally equi-continuous and by Arzel\`a-Ascoli's theorem, and taking into account that $K_D(\gamma_n(T), \gamma_n(S))=|T-S|$ for all $n$, we can extract a subsequence converging on compacta of $\R$ to a geodesic line $\gamma$ such that $\gamma(0)\in K$. 

By Lemma~\ref{lem:vis-land}, $\lim_{t\to\pm \infty}\gamma(t)=p$. Hence, $\gamma$ is a geodesic loop in $\overline{\Omega}$ with vertex  $p$, and we are done.
\end{proof}

For Gromov hyperbolic domains the existence of a geodesic loop is equivalent to the existence of two geodesic rays  landing at the same point which stay at infinite distance each other:
\begin{proposition}\label{Prop:geo-loop-inf}
Let $\Omega\subset\subset \C^d$ be a domain such that $(\Omega, K_\Omega)$ is complete hyperbolic and Gromov hyperbolic. Assume that $\Omega$ is visible. Then there exists a geodesic loop for $\overline{\Omega}$ with vertex $p\in\partial\Omega$ if and only if there exists two geodesic rays in $\Omega$ landing at $p$ which stay at infinite distance each other.
\end{proposition}
\begin{proof}
If $\gamma$ is a geodesic loop for $\overline{\Omega}$ with vertex $p$, then $\gamma^+$ and $\gamma^-$ are geodesic rays at infinite distance each other, landing at $p$ (see Lemma~\ref{Lem:infinite-dist-ray}).

Conversely, if $\alpha, \beta: [0,+\infty)\to \Omega$ are two geodesic rays which stay at infinite distance each other, for each $T>0$, we consider the geodesic rectangle given by
\[
[\alpha(0),\beta(0)]_\Omega\cup \alpha([0,T])\cup [\alpha(T),\beta(T)]_\Omega\cup \beta([0,T]).
\]
Since $(\Omega, K_\Omega)$ is Gromov hyperbolic, by Lemma~\ref{Lem:thick-rect} there exists $\delta>0$ such that $[\alpha(T),\beta(T)]_\Omega$ is contained in the $\delta$-tubular neighborhood (with respect to $K_\Omega$) of $[\alpha(0),\beta(0)]_\Omega\cup \alpha([0,T])\cup \beta([0,T])$.

Since $\alpha$ and $\beta$ stay at infinite distance each other, it follows that there exists $T_0>0$ such that $K_\Omega(\alpha(t),\beta(s))>4\delta$ for all $t,s\geq T_0$. 

We claim that for all $T>T_0$ there exists $z_T\in [\alpha(T),\beta(T)]_\Omega$ such that $z_T$ belongs to the $\delta$-tubular neighborhood $\mathcal N$ of $[\alpha(0),\beta(0)]_\Omega\cup \alpha([0,T_0)]\cup \beta([0,T_0])$. 

Indeed, if this were not the case, then  every $\xi\in  [\alpha(T),\beta(T)]_\Omega$ would belong to a $\delta$-tubular neighborhood $U$ of $\alpha([T_0, T])\cup \beta([T_0, T])$. However, since $K_\Omega(\alpha(t),\beta(s))>4\delta$ for all $t,s\geq T_0$, it follows that $U$ is the disjoint union of two open sets. But then $[\alpha(T),\beta(T)]_\Omega$ is the disjoint union of the two open sets given by $[\alpha(T),\beta(T)]_\Omega\cap U$, hence, not connected, a contradiction and the claim follows. 

Now, since $K_\Omega$ is a complete distance, $K:=\overline{\mathcal N}$ is a compact subset of $\Omega$ and 
\[
[\alpha(T),\beta(T)]_\Omega\cap K\neq\emptyset \quad \forall T\geq 0.
\]
Choosing a parametrization $\gamma_T$ of $[\alpha(T),\beta(T)]_\Omega$ such that $\gamma_T(0)\in K$ and arguing as in the proof of Proposition~\ref{Prop:visibi-no-cont}, we see that we can extract a sequence $\{\gamma_{T_n}\}$ converging to a geodesic loop for $\overline{\Omega}$ with vertex $p$.
\end{proof}

Now we are in good shape to prove Theorem~\ref{Thm:char-geo-loop}: 

\begin{proof}[Proof of Theorem~\ref{Thm:char-geo-loop}]
First, assume that $\Omega$ has no geodesic loops. Then by Proposition~\ref{Prop:vis-frame} (and its proof), the family $\mathcal F$ of all geodesics steaming from a given point $z_0\in\Omega$ is a geodesic frame. Such a geodesic frame is non-looping by Proposition~\ref{Prop:geo-loop-inf}.

Suppose now that $\Omega$ has a non-looping frame of quasi-geodesics $\mathcal F$. Assume by contradiction that $\gamma:\R \to \Omega$ is a geodesic loop. Since $\Omega$ is visible, it follows that there exists $p\in\partial \Omega$ such that $\lim_{t\to\pm \infty} \gamma(t)=p$. Let $\delta>0, \epsilon>0$ be given as in (5) of Definition~\ref{Def:frame} and let 
\[
V:=\{z\in\Omega: \hbox{dist}(z, \partial\Omega)<\epsilon\}.
\] 

Let $\gamma^+(t):=\gamma(t)$ for $t\geq 0$ and $\gamma^-(t):=\gamma(-t)$ for $t\geq 0$. 

Hence, there exists $t_0>0$ such that for $t\geq t_0$ sufficiently large, $\gamma^{\pm}(t)\in V$.  

By hypothesis, for each $k\geq t_0$, $k\in\N$, there exist $\gamma^{\pm}_k\in\mathcal F$ and some $s^{\pm}_k$ in the domain of definition of $\gamma^{\pm}_k$ such that 
\begin{equation}\label{Eq:stay-close+-}
K_\Omega(\gamma^{\pm}(k), \gamma^{\pm}_k(s^{\pm}_k))\leq \epsilon.
\end{equation}
As $k\to+\infty$, since $K_\Omega$ is a complete distance and by Lemma~\ref{Daddezio}, it follows that  $s^{\pm}_k\to +\infty$ and $\gamma^{\pm}_k(s^{\pm}_k)\to p$. By definition of quasi-geodesic frame, up to subsequences, $\{\gamma^{\pm}_k\}$ converges uniformly to some quasi-geodesic ray $\gamma^\pm_\infty\in\mathcal F$ such that 
\[
\lim_{t\to+\infty}\gamma^\pm_\infty(t)=p.
\] 
Since $\mathcal F$ is non-looping, $\gamma_\infty^+$ and $\gamma_\infty^-$ stay at finite distance each other. 

We claim that this implies that there exists $C>0$ such that for all $T\geq 0$,
\begin{equation}\label{Eq:bounded-loop-C-convex}
K_\Omega(\gamma^+(T), \gamma^-(T))\leq C,
\end{equation}
reaching a contradiction to Lemma~\ref{Lem:infinite-dist-ray}. 

In order to prove \eqref{Eq:bounded-loop-C-convex}, 
for every $k\geq t_0$,  consider the quasi-geodesic rectangle given by
\[
[\gamma^+(k), \gamma^+_k(s_k)]_\Omega\cup \gamma^{+}([0,k]) \cup \gamma^{+}_k([0,s_k^+])\cup [\gamma(0), \gamma^{+}_k(0)]_\Omega.
\]
Fix $T\geq t_0$. Hence, by Lemma~\ref{Lem:thick-rect},  there exists some $R>0$ such that for every $k>T$ there exists $z_k\in [\gamma^+(k), \gamma^+_k(s_k)]_\Omega\cup  \gamma^{+}_k([0,s_k^+])\cup [\gamma(0), \gamma^{+}_k(0)]_\Omega$ such that
\begin{equation}\label{Eq:stay-close-gamma+}
K_\Omega(z_k, \gamma^+(T))\leq R.
\end{equation}
Since by hypothesis $\{\gamma^+_k(0)\}$ is relatively compact in $\Omega$---and so is $\{ [\gamma(0), \gamma^{+}_k(0)]_\Omega\}$---and $\gamma^+(t)\to p\in\partial\Omega$ as $t\to +\infty$, it follows that, if $t_0$ (and hence $T$) is large enough, $z_k\not\in [\gamma(0), \gamma^{+}_k(0)]_\Omega$. 

Moreover,  $\{[\gamma^+(k), \gamma^+_k(s_k)]_\Omega\}$ is compactly divergent. Indeed, if this is not the case, there is $y_k\in [\gamma^+(k), \gamma^+_k(s_k)]_\Omega$ such that $\{y_k\}$ is relatively compact, thus
\[
K_\Omega(\gamma^+(k), \gamma^+_k(s_k))=K_\Omega(\gamma^+(k), y_k)+K_\Omega(y_k, \gamma^+_k(s_k)),
\]
and the latter quantity tends to $+\infty$ as $k\to+\infty$ because $K_\Omega$ is a complete distance, against \eqref{Eq:stay-close+-}.

Hence, if $k$ is sufficiently large, $z_k\in \gamma^{+}_k([0,s_k^+])$. By \eqref{Eq:stay-close-gamma+}, $\{z_k\}$ is relatively compact in $\Omega$. 

Since $\gamma_\infty^-$ and $\gamma_\infty^+$ stay at finite distance each other and $\{\gamma_k^\pm\}$ converges uniformly to $\gamma_\infty^\pm$, it follows that there exist $D>0$ and $w_k\in \gamma_k^-$ for $k$ sufficiently large, so that 
\[
K_\Omega(z_k, w_k)\leq D.
\]
In particular, this implies that also $\{w_k\}$ is relatively compact in $\Omega$.

Using again Lemma~\ref{Lem:thick-rect} and arguing as before with the quasi-geodesic rectangle 
\[
[\gamma^-(k), \gamma^-_k(s_k)]_\Omega\cup \gamma^{-}([0,k]) \cup \gamma^{-}_k([0,s_k^-])\cup [\gamma(0), \gamma^{-}_k(0)]_\Omega,
\]
we can find  $S\geq 0$ and $E>0$ such that for all $k$ sufficiently large
\[
K_\Omega(\gamma^-(S), w_k)\leq E.
\]
Then, by the triangle inequality, $K_\Omega(\gamma^+(T), \gamma^-(S))\leq C:=R+D+E$. Therefore,
\begin{equation*}
\begin{split}
K_\Omega(\gamma^+(T), \gamma^-(T))&\leq K_\Omega(\gamma^+(T), \gamma^-(S))+K_\Omega(\gamma^-(S), \gamma^-(T))\\&=K_\Omega(\gamma^+(T), \gamma^-(S))+|T-S|\\&=K_\Omega(\gamma^+(T), \gamma^-(S))+|K_\Omega(\gamma^+(T), \gamma(0))-K_\Omega(\gamma^-(S), \gamma(0))|\\
&\leq 2 K_\Omega(\gamma^+(T), \gamma^-(S))\leq C.
\end{split}
\end{equation*}
Thus, \eqref{Eq:bounded-loop-C-convex} follows, and we are done.
\end{proof}

\end{document}